\documentclass[11pt, a4paper, reqno]{amsart} 
\usepackage[left=2.5cm, right=2.5cm, bottom=3cm]{geometry}
\usepackage{graphicx} 
\usepackage{amsmath}
\usepackage{amssymb}
\usepackage{amsthm}
\usepackage{graphicx}
\usepackage{subcaption}
\usepackage{float}
\usepackage{bbm}

\usepackage[dvipsnames]{xcolor}

\newtheorem{theorem}{Theorem}
\newtheorem{lemma}{Lemma}

\newcommand{\innerproduct}[2]{\langle #1, #2 \rangle}

\DeclareMathOperator{\spn}{span}

\title{Time-frequency localization in the Fourier Symmetric Sobolev space}
\author{Denis Zelent}
\address{Department of Mathematical Sciences, Norwegian University of Science and Technology (NTNU), 7491 Trondheim, Norway} 
\email{denis.zelent@ntnu.no}

\thanks{The author was supported by Grant 334466 of the 
Research Council of Norway.}

\begin{document}

\begin{abstract}
    We study concentration operators acting on the Fourier symmetric Sobolev space~$\mathcal{H}$ consisting of functions $f$ such that
    $\int_{\mathbb{R}} |f(x)|^2(1+x^2) dx +  \int_{\mathbb{R}} |\hat{f}(\xi)|^2(1+\xi^2) d\xi < \infty  $.
    We find that the Bargmann transform is a unitary operator from $\mathcal{H}$ to a weighted Fock space. After identifying the reproducing kernel of $\mathcal{H}$, we discover an unexpected phenomenon about the decay of the eigenvalues of a two-sided concentration operator, namely that the plunge region is of the same order of magnitude as the region where the eigenvalues are close to 1, contrasting the classical case of Paley--Wiener spaces. 
\end{abstract}
\maketitle

\section{Introduction}
The Fourier symmetric Sobolev space $\mathcal{H}$ is defined to be 
$$\mathcal{H} = \{f: f,\hat{f} \in \mathcal{H}_1, \| f \|^2_{\mathcal{H}} = \| f \|^2_{\mathcal{H}_1} + \| \hat{f} \|^2_{\mathcal{H}_1} \},$$ 
where $\mathcal{H}_1$ is the Sobolev space of functions $f$ on $\mathbb{R}$ such that
    $$\| f \|^2_{\mathcal{H}_1} = \int_{\mathbb{R}} |\hat{f}(\xi)|^2(1+\xi^2) d\xi < \infty,
    $$
with $\hat{f}(t) = \mathcal{F}f(t)=\int_{\mathbb{R}}f(x)e^{-2\pi itx}dx.$
Since $\mathcal{H}$ treats time and frequency in a symmetric way and is invariant under the Fourier transform, it may reflect interesting aspects of the uncertainty principle. For example, the Balian--Low theorem states that there is no function $g\in \mathcal{H}$ such that the Gabor system $\{e^{2\pi i mbx}g(x-na)\}_{m,n \in \mathbb{Z}}$ with $ab=1$ forms an orthonormal basis for $L^2(\mathbb{R})$ (see e.g. \cite[§8.4]{gr01}). 
In addition, $\mathcal{H}$ provides a natural setting for an analytic study of Fourier interpolation, as shown in the recent paper \cite[Thm.~1.3 and the following discussion]{kulikov2023fourier}.
It also appears in the study of pseudodifferential operators as one of the Shubin--Sobolev spaces (also called Shubin classes), which turns out to be one of the modulation spaces on $\mathbb{R}$, namely
\begin{align*}
    M_{v_1}^{2}(\mathbb{R}) = \left\{f\in \mathcal{S}'(\mathbb{R}) : ||f||_{M_{v_1}^{2}} = \left(\int_{\mathbb{R}}\int_{\mathbb{R}} |V_gf(x,\omega)|^2 (1+x^2+\omega^2)dxd\omega\right)^{1/2} < \infty \right\},
\end{align*}
where $g(t) = 2^{1/4}e^{-\pi t^2}$ is the normalized gaussian and 
\begin{align*}
    V_gf(x,\omega) = \int_{\mathbb{R}}f(t)\overline{g(t-x)}e^{-2\pi i \omega t}dt
\end{align*}
is the definition of the short-time Fourier transform, see e.g. \cite[§2.2]{boggiatto2004},\cite{boto05}, \cite{Luef11} and \cite[§25.3]{shubin}.

It is known from \cite{kulikov2023fourier} that $\mathcal{H}$ is a reproducing kernel Hilbert space, but the kernel has not been identified. 
We will find an explicit formula for this kernel based on the discovery that the scaled Hermite functions form an orthonormal basis of $\mathcal{H}$. Another proof of this, using that the Hermite functions are eigenfunctions of the Hamiltonian operator of the classical quantum harmonic oscillator, appeared in \cite{szehr25} after the announcement of the first version of the present paper. We obtain this result by showing that the Bargmann transform is a unitary operator from $\mathcal{H}$ to the weighted Fock space with
$$||F||^2 = \int_\mathbb{C}|F(z)|^2e^{-2|z|^2}\left(2\pi-\frac{1}{2}+2|z|^{2}\right)dA <\infty.$$

We will use the reproducing kernel $K_t$ of $\mathcal{H}$ to study the problem of maximizing the ratio
\begin{align}\label{conce}
    \frac{\int_{I} |f(x)|^2(1+x^2) dx+\int_{J} |\hat{f}(\xi)|^2(1+\xi^2) d\xi}{\int_{\mathbb{R}} |f(x)|^2 (1+x^2)dx+\int_{\mathbb{R}} |\hat{f}(\xi)|^2(1+\xi^2) d\xi},
\end{align}
where $I,J\subset \mathbb{R}$ are measurable.
The inspiration for our study is the classical Paley--Wiener case (see e.g. \cite{KULIKOV2024101639,Landau,LANDAU1980469,Slepian}), where one seeks to maximize 
\begin{align*}
    \frac{\int_{I} |f(x)|^2 dx}{\int_{\mathbb{R}} |f(x)|^2 dx}
\end{align*}
for $f$ in the Paley--Wiener space. By standard arguments, we are in this case led to consider the eigenvalues of the Hilbert--Schmidt operator $S_I$ defined by $S_If(t) = \int_I f(x)\overline{k_t(x)}dx$, where $k_t(x)$ is the reproducing kernel of the Paley--Wiener space. If $\lambda_n(I)$ denotes the eigenvalues of $S_I$ in decreasing order, one obtains $\sum_n \lambda_n(I) \sim \sum_n \lambda_n^2(I)$ as $|I|\to\infty$, where as usual $f(x)\sim g(x)$ means $\lim_{x\to\infty} \frac{f(x)}{g(x)}=1.$
This implies that the eigenvalues are at first very close to $1$, then there is a short plunge region where they decay, and afterwards they are very close to $0$. This behavior allows us, for example, to establish critical densities for sampling and interpolation, as was done in ~\cite{Landau}. 

In our case, the problem of maximizing (\ref{conce}) leads to the consideration of the eigenvalues of the operator $T_{I,J}:\mathcal{H} \to \mathcal{H}$ defined as
$$T_{I,J}f(t) = \int_{I} f(x)\overline{K_t(x)}(1+x^2)dx + \int_{J} \hat{f}(x)\overline{\hat{K_t}(x)}(1+x^2)dx.$$
Interestingly, we discover that in this case, the plunge region is much longer. Indeed, if $\lambda_n$ denotes the $n$th eigenvalue of $T_{I,I}$ for $I=[-R,R]$, then $\lambda_n\in(0,1)$ and
\begin{align*}
    \sum_n \lambda_n &\sim 2\pi R^2, \quad \sum_n \lambda_n^2 \sim  (2\pi-2)R^2.
\end{align*}
We have not found a complete description of the asymptotic behavior of these eigenvalues, but the following theorem shows that many of them will cluster close to $1$.
\begin{theorem} \label{4R^2th}
    Fix $\epsilon>0$ and let $I=[-R,R], J=[-T,T]$. Then there are
    \begin{itemize}
        \item at least $4RT - O(\log R)-O(T)$ eigenvalues of $T_{I,J}$ larger than $1-\epsilon$,
        \item at most $2\pi(R^2+T^2) + o(R^2)+o(T^2)$ eigenvalues of $T_{I,J}$ larger than $\epsilon$,
    \end{itemize}
    where $R,T\to \infty$.
\end{theorem}
We conjecture that the number of these eigenvalues larger than $1-\epsilon$ is also at most $4RT+o(RT)$. Using the $4RT - O(\log R)-O(T)$ eigenvalues of $T_{I,J}$ which are close to $1$, we may be able to improve the second part of the theorem to at most $2\pi(R^2+T^2) - 4RT + o(R^2)+o(T^2)$ eigenvalues of $T_{I,J}$ larger than $\epsilon$. However, the proof would need to be more complicated, and there is no reason to believe that the new bound would be optimal.

To prove Theorem \ref{4R^2th}, we consider the analogous problem of concentration in the Paley--Wiener space $PW_{T}(1+x^2)$ of functions $f$ with Fourier transform supported in $[-T, T]$ and with the norm
$$||f||^2 = \int_{\mathbb{R}} |f(x)|^2 (1+x^2)dx<\infty.$$
Here we are interested in maximizing the ratio 
\begin{align*}
    \frac{\int_{-R}^R |f(x)|^2 (1+x^2)dx}{\int_{\mathbb{R}} |f(x)|^2 (1+x^2)dx},
\end{align*}
which leads to considering the eigenvalues of the operator $W_R:PW_{T}(1+x^2)\to PW_{T}(1+x^2)$ defined as    
\begin{align*}
    W_Rf(t) = \int_{-R}^R f(x)\overline{L_t(x)}(1+x^2)dx,
\end{align*}
where $L_t$ is the reproducing kernel of $PW_{T}(1+x^2)$. For this operator, we obtain
\begin{theorem}\label{PWcon}
    For any $\epsilon>0$ there are $4RT + O(\log R)+O(T)$ eigenvalues of $W_R$ larger than $1-\epsilon$ as $R,T\to\infty$.
\end{theorem}
Finally, we consider a one-sided concentration problem in $\mathcal{H}$, namely that of maximizing the ratio
\begin{align}\label{conce1}
    \frac{\int_{I} |f(x)|^2(1+x^2) dx}{\int_{\mathbb{R}} |f(x)|^2(1+x^2) dx+\int_{\mathbb{R}} |\hat{f}(\xi)|^2(1+\xi^2) d\xi},
\end{align}
where $I\subset \mathbb{R}$ measurable.
In this case, for the eigenvalues of the corresponding operator $T_I:~\mathcal{H} \to \mathcal{H}$ defined as
$$T_If(t) = \int_{I} f(x)\overline{K_t(x)}(1+x^2)dx$$
we get
\begin{theorem} \label{onesideth}
    Fix $0<\epsilon<\frac{1}{4}$ and let $I=[-R,R]$. Then there are 
    \begin{itemize}
        \item at most $o(R^2)$ eigenvalues of $T_I$ larger than $0.5+\epsilon$,
        \item $2\pi R^2 + o(R^2)$ eigenvalues of $T_I$ between $0.5-\epsilon$ and $0.5+\epsilon$,
        \item at most $o(R^2)$ eigenvalues of $T_I$ between $\epsilon$ and $0.5-\epsilon$,
    \end{itemize}
    where $R\to\infty.$
\end{theorem}
It is interesting that most of the eigenvalues of $T_I$ larger than $\epsilon$ cluster around $\frac{1}{2}$, since, in general, $\int_{\mathbb{R}} |f(x)|^2(1+x^2) dx$ and $\int_{\mathbb{R}} |\hat{f}(\xi)|^2(1+\xi^2) d\xi$ may differ considerably for a ``generic'' $f\in \mathcal{H}$. Theorem \ref{onesideth} also shows that the plunge region for the eigenvalues of $T_I$ is negligible. 

The results of this paper exhibit a clear distinction from the well-studied Gabor--Toeplitz localization operator, which for a nonnegative, bounded function $a$ on $\mathbb{R}^2$ is defined as 
\begin{align*}
    T_af(t) = \iint\limits_{\mathbb{R}^2} a(x,\omega) V_gf(x,\omega) g(t-x) e^{2\pi it\omega}dxd\omega.
\end{align*}
See e.g. \cite{Fei02, Fei01, Fei14, Gro13} for the study of $T_a$.
When, in particular, $a$ is the indicator function of $[-R,R]\times [-T,T]$, Theorem 4.1 in \cite{Fei02} shows that for the Gabor--Toeplitz localization operator $T_a$, the size of the plunge region is small compared with the size of the region where the eigenvalues are close to 1. Note that the concentration operator $T_{[-R,R],[-T,T]}$ studied in this paper also measures the time-frequency concentration of $f$ on $[-R,R]\times [-T,T]$ in accordance with (\ref{conce}), but produces a different eigenvalue behaviour. 

Theorem \ref{onesideth} sheds some light on this difference. Consider, for example, $I=[-R,R]$, $J=[-\epsilon,\epsilon]$ for some fixed $R$. By decreasing the $\epsilon$, the eigenvalues of the Gabor--Toeplitz operator with $a$ being the indicator function of $I\times J$ will necessarily decrease to $0$. At the same time, 
\begin{align*}
    \frac{\int_{I} |f(x)|^2(1+x^2) dx+\int_{J} |\hat{f}(\xi)|^2(1+\xi^2) d\xi}{\int_{\mathbb{R}} |f(x)|^2 (1+x^2)dx+\int_{\mathbb{R}} |\hat{f}(\xi)|^2(1+\xi^2) d\xi} \geq \frac{\int_{I} |f(x)|^2(1+x^2) dx}{\int_{\mathbb{R}} |f(x)|^2 (1+x^2)dx+\int_{\mathbb{R}} |\hat{f}(\xi)|^2(1+\xi^2) d\xi}, 
\end{align*}
and by Theorem \ref{onesideth} there are many eigenvalues of $T_{I,J}$ around $0.5$.

The paper is organized as follows. In Section 2, we find the closed formula for the reproducing kernel of $\mathcal{H}$. In Section 3, we discuss the concentration problem (\ref{conce}). In Section 4, we turn our attention to the concentration operator in the weighted Paley--Wiener space $PW_{T}(1+x^2)$. In Section 5, we return to the space $\mathcal{H}$ and examine the concentration problem (\ref{conce1}). Finally, in Section 6, we prove Theorem \ref{4R^2th}.

\section{The reproducing kernel of $\mathcal{H}$}\label{rk}
In this section, we find the reproducing kernel of $\mathcal{H}$, which was done as part of the author's master's thesis \cite{master}. For the convenience of the reader, we include this computation here.

Since a reproducing kernel can be expressed using an orthonormal basis, we want to find a suitable orthonormal basis of $\mathcal{H}$. We will show that such a basis consists of scaled Hermite functions
$e^{-\pi x^2}H_n(\sqrt{2\pi}x),$ where $H_n$ denotes the $n$th Hermite polynomial.
We begin with the following lemma.
\begin{lemma}\label{ortho}
If $n>m+2$, $n,m\geq 0$, then
    $$\int_{\mathbb{R}}H_n(\sqrt{2\pi}x)H_m(\sqrt{2\pi}x)e^{-2\pi x^2}(1+x^2)dx = 0.$$
\end{lemma}
\begin{proof}
    The lemma follows by using $H_n(x) = (-1)^n e^{x^2}\frac{d^n}{dx^n}e^{-x^2}$ and $n$ integrations by parts:
\begin{align*}
    \int_{\mathbb{R}}H_n(\sqrt{2\pi}x)H_m(\sqrt{2\pi}x)e^{-2\pi x^2}(1+x^2)dx
    &=\frac{1}{\sqrt{2\pi}}\int_{\mathbb{R}}(-1)^n \left(\frac{d^n}{dx^n}e^{-x^2}\right)H_m(x)\left(1+\frac{x^2}{2\pi}\right)dx \\
    &= \frac{1}{\sqrt{2\pi}} \int_{\mathbb{R}} \left(\frac{d^n}{dx^n}H_m(x)\left(1+\frac{x^2}{2\pi}\right)\right)e^{-x^2}dx = 0,
\end{align*}
     as $H_m(x)$ is a polynomial of degree $m$, and therefore $\frac{d^n}{dx^n}H_m(x)\left(1+\frac{x^2}{2\pi}\right) = 0$ for $n>m+2$.
\end{proof}

\begin{lemma}\label{BH_basis}
    The Hermite functions $H_n(\sqrt{2\pi}x) e^{-\pi x^2}$ are pairwise orthogonal in $\mathcal{H}$.
\end{lemma}
\begin{proof}
    We will use the fact that the Hermite functions are eigenvalues of the Fourier transform: $\mathcal{F}(H_n(\sqrt{2\pi}x)e^{-\pi x^2})(\xi) = (-i)^n H_n(\sqrt{2\pi}\xi)e^{-\pi \xi^2}$. For any $n, m\geq 0$ we have
\begin{align*}
    \innerproduct{H_n(\sqrt{2\pi}x)&e^{-\pi x^2}}{H_m(\sqrt{2\pi}x)e^{-\pi x^2}}_{\mathcal{H}} \\
    &=\int_{\mathbb{R}}H_n(\sqrt{2\pi}x)e^{-\pi x^2}\overline{H_m(\sqrt{2\pi}x)e^{-\pi x^2}}(1+x^2)dx \\
    & +\int_{\mathbb{R}}\mathcal{F}(H_n(\sqrt{2\pi}x)e^{-\pi x^2})(\xi)\overline{\mathcal{F}(H_m(\sqrt{2\pi}x)e^{-\pi x^2})(\xi)}(1+\xi^2)d\xi \\
    &= \int_{\mathbb{R}}H_n(\sqrt{2\pi}x)H_m(\sqrt{2\pi}x)e^{-2\pi x^2}(1+x^2)dx + \\
    & +\int_{\mathbb{R}}(-i)^nH_n(\sqrt{2\pi}\xi)e^{-\pi \xi^2}(\xi)\overline{(-i)^m}H_m(\sqrt{2\pi}\xi)e^{-\pi \xi^2}(1+\xi^2)d\xi \\
    &= \int_{\mathbb{R}}H_n(\sqrt{2\pi}x)H_m(\sqrt{2\pi}x)e^{-2\pi x^2}(1+x^2)dx (1+(-i)^n i^m).
\end{align*}
    Note that if $n,m$ differs by one, then $H_nH_m$ is an odd function, and so the integral is zero. If $n,m$ differ by two, then $(1+(-i)^n i^m)=0$. If $n,m$ differ by more than two, then the integral is zero by Lemma \ref{ortho}. Thus, if $n\neq m$, then $\innerproduct{H_n(\sqrt{2\pi}x)e^{-\pi x^2}}{H_m(\sqrt{2\pi}x)e^{-\pi x^2}}_{\mathcal{H}} = 0$.
\end{proof}
To normalize the functions, we need to find the norms of $H_n(\sqrt{2\pi}x) e^{-\pi x^2}$.
\begin{lemma}\label{norm}
    For every $n \geq 0$,
    $$\int_{\mathbb{R}} H_n^2(\sqrt{2\pi}x)e^{-2\pi x^2}(1+x^2)dx = 2^{n-\frac{3}{2}}  n!\left(n+2\pi+\frac{1}{2}\right)\frac{1}{\pi}.$$
Thus 
    $$\| H_n(\sqrt{2\pi}x)e^{-\pi x^2}\|^2_\mathcal{H} = 2^{n-\frac{1}{2}}  n!\left(n+2\pi+\frac{1}{2}\right)\frac{1}{\pi}.$$
\end{lemma}
\begin{proof}
    Notice first, that using again $H_n(x) = (-1)^n e^{x^2}\frac{d^n}{dx^n}e^{-x^2}$ and $n$ integrations by parts, we obtain
\begin{align*}
    \int_{\mathbb{R}}x^n H_n(\sqrt{2\pi}x)e^{-2\pi x^2}dx &= (2\pi)^{-\frac{n+1}{2}}\int_{\mathbb{R}}x^n H_n(x)e^{-x^2}dx=(2\pi)^{-\frac{n+1}{2}} \int_{\mathbb{R}}x^n (-1)^n \frac{d^n}{dx^n}e^{-x^2}dx \\
    &= (2\pi)^{-\frac{n+1}{2}} n! \int_{\mathbb{R}}e^{-x^2}dx = (2\pi)^{-\frac{n+1}{2}} n!\sqrt{\pi},
\end{align*}
\begin{align*}
    \int_{\mathbb{R}}x^{n+2} H_n(\sqrt{2\pi}x)e^{-2\pi x^2}dx &
    =(2\pi)^{-\frac{n+3}{2}}\int_{\mathbb{R}}x^{n+2} (-1)^n \frac{d^n}{dx^n}e^{-x^2}dx \\
    &= (2\pi)^{-\frac{n+3}{2}}\frac{(n+2)!}{2!} \int_{\mathbb{R}}x^2e^{-x^2}dx = (2\pi)^{-\frac{n+3}{2}}\frac{(n+2)!}{2}\frac{\sqrt{\pi}}{2},
\end{align*}
    for any $m<n$ 
    $$\int_{\mathbb{R}}x^m H_n(\sqrt{2\pi}x)e^{-2\pi x^2}dx = 0.$$
    Combining this with the fact that $H_n$ can be written as 
    $$H_n(\sqrt{2\pi}x) = n!\sum_{m=0}^{\lfloor n/2 \rfloor}\frac{(-1)^m}{m!(n-2m)!}(2\sqrt{2\pi}x)^{n-2m},$$
    we have
\begin{align*}
    &\int_{\mathbb{R}} H_n^2(\sqrt{2\pi}x)e^{-2\pi x^2}(1+x^2)dx \\
    &= \int_{\mathbb{R}} \left((2\sqrt{2\pi})^nx^n - \frac{n!}{(n-2)!}(2\sqrt{2\pi}x)^{n-2}\right)H_n(\sqrt{2\pi}x)e^{-2\pi x^2}(1+x^2)dx \\
    &= (2\sqrt{2\pi})^n (2\pi)^{-\frac{n+1}{2}}n!\sqrt{\pi} + 
    (2\sqrt{2\pi})^n (2\pi)^{-\frac{n+3}{2}}\frac{(n+2)!}{2}\frac{\sqrt{\pi}}{2} - \frac{(2\sqrt{2\pi})^{n-2}n!}{(n-2)!}(2\pi)^{-\frac{n+1}{2}}n!\sqrt{\pi} \\
    &= 2^n (2\pi)^{-\frac{3}{2}} n!\sqrt{\pi}\left(n+2\pi+\frac{1}{2}\right).
\end{align*}
\end{proof}
It remains to show that the Hermite functions constitute a complete sequence in $\mathcal{H}$.
The idea of doing so is based on \cite{lyubarskii_convergence_1999}. To this end we introduce the Bargmann transform
$$\mathfrak{B}: f\to F(z) = (\mathfrak{B}f)(z) = \frac{2^{1/4}}{\pi^{3/2}}\int_{\mathbb{R}} f\left(\frac{t}{\sqrt{2\pi}}\right)e^{2tz-z^2-t^2/2}dt.$$
If we let $\mathcal{B}_\beta$ be the space of all entire functions satisfying 
$$||F||^2_{\mathcal{B}_\beta} = \int_\mathbb{C}|F(z)|^2e^{-2|z|^2}\left(2\pi-\frac{1}{2}+2|z|^{2}\right)^{\beta}dA < \infty,$$
where $dA$ denote the Lebesgue area measure on $\mathbb{C}$, then by \cite{lyubarskii_convergence_1999} $\mathfrak{B}$ is, up to a constant factor, a unitary linear operator from $L^2(\mathbb{R})$ onto $\mathcal{B}_0$, and a bounded invertible mapping from $\mathcal{H}$ onto $\mathcal{B}_1$. We will improve this result by showing that $\mathfrak{B}$ is also a unitary operator from $\mathcal{H}$ onto $\mathcal{B}_1$. 
Proving that it sends an orthonormal basis of $\mathcal{B}_1$ to the orthonormal set of Hermite functions, we will conclude the proof.\\
\textbf{Remark.} It is worth noting that the above are the only two cases when the Bargmann transform can be a unitary operator between the Schwartz scale $\chi_\beta$ consisting of functions such that
\begin{align*}
    ||f||^2_{\chi_\beta} = \int_{\mathbb{R}} |f(x)|^2(1+x^{2\beta}) dx +  \int_{\mathbb{R}} |\hat{f}(\xi)|^2(1+\xi^{2\beta}) d\xi < \infty,
\end{align*}
and the corresponding space $\mathcal{B}_\beta$.

\begin{lemma}\label{basFock}
    The system $\left\{\sqrt{\frac{2^{n+1}}{n!(n+2\pi+1/2)\pi}} z^n, n\geq 0\right\}$ forms an orthonormal basis of $\mathcal{B}_1$.
\end{lemma}
\begin{proof}
    Orthogonality is a trivial computation, while completeness follows from analyticity of functions in $\mathcal{B}_1$ (any $f\in \mathcal{B}_1$ can be written as $f(z) = \sum_{n\geq 0} a_n z^n$, and so $\innerproduct{f}{z^m}_{\mathcal{B}_1}=0$ for every $m$ implies that $a_m=0$ for every $m$). Now,
    \begin{align*}
        \|z^n\|^2_{\mathcal{B}_1} &= \int_\mathbb{C}|z|^{2n}e^{-2|z|^2}\left(2\pi-\frac{1}{2}+2|z|^{2}\right)dA = 2\pi \int_0^{\infty} r^{2n+1} e^{-2r^2} \left(2\pi-\frac{1}{2}+2r^{2}\right)dr \\
        &= 2\pi \int_0^{\infty} (t/2)^{n+\frac{1}{2}} e^{-t} (2\pi-\frac{1}{2}+t) \frac{dt}{4\sqrt{t/2}} = \frac{2\pi}{4\cdot 2^n} \int_0^{\infty} t^{n} e^{-t} (2\pi-\frac{1}{2}+t) dt \\
        &= \frac{\pi}{2^{n+1}} \left(\left(2\pi-\frac{1}{2}\right)n! + (n+1)!\right) = \frac{\pi n!}{2^{n+1}} (n +2\pi+ 1/2),
    \end{align*}
    which concludes the proof.
\end{proof}

\begin{lemma}\label{bas}
    We have
    $$\mathfrak{B}(H_n(\sqrt{2\pi}x)e^{-\pi x^2})(z) = \frac{\sqrt{2}}{\pi^{1/4}} (2z)^n, $$
    and therefore 
    $$\mathfrak{B}^{-1}\left(\sqrt{\frac{2^{n+1}}{n!(n+2\pi+1/2)\pi}} z^n\right) = \sqrt{\frac{\pi}{2^{n-\frac{1}{2}}  n!\left(n+2\pi+\frac{1}{2}\right)}} H_n(\sqrt{2\pi}x)e^{-\pi x^2}.$$
\end{lemma}
\begin{proof}
    The first part follows by $H_n(x) = (-1)^n e^{x^2}\frac{d^n}{dx^n}e^{-x^2}$ and $n$ integrations by parts:
\begin{align*}
    \mathfrak{B}(H_n(\sqrt{2\pi}x)e^{-\pi x^2})(z) &= \frac{2^{1/4}}{\pi^{3/2}}\int_{\mathbb{R}} H_n(x)e^{-x^2/2} e^{2xz-z^2-x^2/2}dx \\ 
    &= \frac{2^{1/4}}{\pi^{3/2}}\int_{\mathbb{R}} (-1)^n\left( \frac{d^n}{dx^n}e^{-x^2}\right) e^{2xz-z^2}dx \\ 
    &=\frac{2^{1/4}}{\pi^{3/2}}\int_{\mathbb{R}} e^{-x^2} (2z)^n e^{2xz-z^2}dx  \\
    &=  \frac{2^{n+1/4}}{\pi} z^n.
\end{align*}
    This means that 
    $$\mathfrak{B}^{-1}(z^n) =  
     \frac{\pi}{2^{n+1/4}}e^{-\pi x^2}H_n(\sqrt{2\pi}x), $$
     and thus
 \begin{align*}
     \mathfrak{B}^{-1}\left(\sqrt{\frac{2^{n+1}}{n!(n+2\pi+1/2)\pi}} z^n\right) &=  
     \sqrt{\frac{2^{n+1}}{n!(n+2\pi+1/2)\pi}}\frac{\pi}{2^{n+1/4}}e^{-\pi x^2}H_n(\sqrt{2\pi}x) \\
     &= \sqrt{\frac{\pi}{2^{n-\frac{1}{2}}  n!\left(n+2\pi+\frac{1}{2}\right)}} H_n(\sqrt{2\pi}x)e^{-\pi x^2}.
 \end{align*}
\end{proof}
\begin{theorem}\label{ONB}
    The Bargmann transform is a unitary operator from $\mathcal{H}$ onto $\mathcal{B}_1$.
    Thus, the system $\left\{\sqrt{\frac{\pi}{2^{n-\frac{1}{2}}  n!\left(n+2\pi+\frac{1}{2}\right)}} H_n(\sqrt{2\pi}x)e^{-\pi x^2}, n\geq 0 \right\}$ forms an orthonormal basis of $\mathcal{H}$.
\end{theorem}
\begin{proof}
    Our lemmas prove that for any $n,m \geq 0$,
\begin{align*}
    &\left\langle \sqrt{\frac{2^{n+1}}{n!(n+2\pi+1/2)\pi}}z^n,\sqrt{\frac{2^{m+1}}{m!(m+2\pi+1/2)\pi}}z^m\right\rangle_{\mathcal{B}_1} \\
    &=\left\langle \mathfrak{B}^{-1}\left(\sqrt{\frac{2^{n+1}}{n!(n+2\pi+1/2)\pi}}z^n\right),\mathfrak{B}^{-1}\left(\sqrt{\frac{2^{m+1}}{m!(m+2\pi+1/2)\pi}}z^m\right)\right\rangle_{\mathcal{H}}.
\end{align*}
Since $\left\{\sqrt{\frac{2^{n+1}}{n!(n+2\pi+1/2)\pi}}z^n, n\geq 0\right\}$ forms an orthonormal basis of $\mathcal{B}_1$, this proves that $\mathfrak{B}^{-1}$ is a unitary transformation $\mathcal{B}_1 \to \mathcal{H}$. By lemma \ref{bas}, since a unitary map sends an orthonormal basis to an orthonormal basis, $\left\{\sqrt{\frac{\pi}{2^{n-\frac{1}{2}}  n!\left(n+2\pi+\frac{1}{2}\right)}} H_n(\sqrt{2\pi}x)e^{-\pi x^2}, n\geq 0 \right\}$ is an orthonormal basis for $\mathcal{H}$.
\end{proof}
Having obtained an orthonormal basis, we can now turn our attention to finding the reproducing kernel of $\mathcal{H}$.
\begin{theorem}\label{repker}
    The reproducing kernel $K_x(y)$ of $\mathcal{H}$ is given by
    $$K_x(y) = \sqrt{2} \pi e^{-\pi(x^2+y^2)} \int_0^1 \frac{t^{2\pi-1/2}}{\sqrt{1-t^2}} e^{2\pi \frac{2xyt-(x^2+y^2)t^2}{1-t^2}} dt.$$
\end{theorem}
\begin{proof}
    Using the orthonormal basis of $\mathcal{H}$ from Theorem \ref{ONB} we find that 
\begin{align*}
    K_x(y) &= \sum_{n=0}^{\infty} \frac{H_n(\sqrt{2\pi}x)e^{-\pi x^2} H_n(\sqrt{2\pi}y)e^{-\pi y^2} \pi}{2^{n-\frac{1}{2}}  n!\left(n+2\pi+\frac{1}{2}\right)}  \\ 
    &= \sqrt{2} \pi e^{-\pi(x^2+y^2)}\sum_{n=0}^{\infty} \frac{H_n(\sqrt{2\pi}x)H_n(\sqrt{2\pi}y)}{2^{n}n!(n+2\pi+1/2)}.
\end{align*}
    To get the integral representation we will use Mehler's Hermite polynomial formula \cite{Watson},
    $$\sum_{n=0}^{\infty} \frac{H_n(\sqrt{2\pi}x)H_n(\sqrt{2\pi}y)}{2^{n}n!} t^n = (1-t^2)^{-1/2} e^{2\pi \frac{2xyt-(x^2+y^2)t^2}{1-t^2}}.$$
    Multiplying both sides by $t^{2\pi-1/2}$ and then integrating from 0 to 1 gives
    $$\sum_{n=0}^{\infty} \frac{H_n(\sqrt{2\pi}x)H_n(\sqrt{2\pi}y)}{2^{n}n!(n+2\pi+1/2)} = \int_0^1 \frac{t^{2\pi-1/2}}{\sqrt{1-t^2}} e^{2\pi \frac{2xyt-(x^2+y^2)t^2}{1-t^2}} dt,$$
    which concludes the proof.
\end{proof}
The Fourier transform of the kernel, which will also be used later, can be found in a similar way and is given in the following theorem.
\begin{theorem}\label{ftrepker}
    The Fourier transform of the reproducing kernel $K_x(y)$ of $\mathcal{H}$ is given by
    $$\widehat{K_x}(y) = \sqrt{2} \pi e^{-\pi(x^2+y^2)} \int_0^1 \frac{t^{2\pi-1/2}}{\sqrt{1+t^2}} e^{2\pi \frac{-2xyti+(x^2+y^2)t^2}{1+t^2}} dt.$$
\end{theorem}
\begin{proof}
    First, we will again use the facts that $$K_x(y) = \sqrt{2} \pi e^{-\pi(x^2+y^2)}\sum_{n=0}^{\infty} \frac{H_n(\sqrt{2\pi}x)H_n(\sqrt{2\pi}y)}{2^{n}n!(n+2\pi+1/2)},$$
    and that the Fourier transform of 
    $e^{-\pi x^2}H_n(\sqrt{2\pi}x)$ is equal to $(-i)^n e^{-\pi \xi^2}H_n(\sqrt{2\pi}\xi).$
    This gives 
    $$\widehat{K_x}(y) = \sqrt{2} \pi e^{-\pi(x^2+y^2)}\sum_{n=0}^{\infty} (-i)^n\frac{H_n(\sqrt{2\pi}x)H_n(\sqrt{2\pi}y)}{2^{n}n!(n+2\pi+1/2)}.$$
    Multiplying the Mehler’s Hermite polynomial formula by $t^{2\pi - 1/2}$ as before, but now integrating from $0$ to $-i$, we obtain
\begin{align*}
    \widehat{K_x}(y) &= \sqrt{2} \pi e^{-\pi(x^2+y^2)} (-i)^{-2\pi-1/2}\int_0^{-i} \frac{t^{2\pi-1/2}}{\sqrt{1-t^2}} e^{2\pi \frac{2xyt-(x^2+y^2)t^2}{1-t^2}} dt\\
    &= \sqrt{2} \pi e^{-\pi(x^2+y^2)} \int_0^1 \frac{t^{2\pi-1/2}}{\sqrt{1+t^2}} e^{2\pi \frac{-2xyti+(x^2+y^2)t^2}{1+t^2}} dt,
\end{align*}
    where in the last equality we used the change of variable $t \to -it$.
\end{proof}

\section{Eigenvalue problem for two-sided concentration}
Knowing the expression for the reproducing kernel allows us to study various concentration problems. The most natural one for the space $\mathcal{H}$ is the problem of maximizing the ratio
\begin{align*}
    \lambda = \frac{\int_{I} |f(x)|^2(1+x^2) dx+\int_{J} |\hat{f}(\xi)|^2(1+\xi^2) d\xi}{\int_{\mathbb{R}} |f(x)|^2 (1+x^2)dx+\int_{\mathbb{R}} |\hat{f}(\xi)|^2(1+\xi^2) d\xi}
\end{align*}
where $I,J\subset \mathbb{R}$ are measurable, which leads to the eigenvalue problem
$$\lambda_n f_n(t) = \int_{I} f_n(x)\overline{K_t(x)}(1+x^2)dx + \int_{J} \hat{f_n}(\xi)\overline{\hat{K_t}(\xi)}(1+\xi^2)d\xi,$$
where $K_t$ is the reproducing kernel of $\mathcal{H}$.
We define thus an operator $T_{I,J}:\mathcal{H} \to \mathcal{H}$ as
$$T_{I,J}f(t) = \int_{I} f(x)\overline{K_t(x)}(1+x^2)dx + \int_{J} \hat{f}(x)\overline{\hat{K_t}(x)}(1+x^2)dx.$$
We want to show that $T_{I,J}$ is a self-adjoint Hilbert--Schmidt operator and study its eigenvalues, which corresponds to the next best possible concentrations $\lambda$ obtained by orthogonal functions in the space.
\begin{lemma}
    For any measurable $I,J$ of finite Lebesgue measure, $T_{I,J}$ is a self-adjoint Hilbert--Schmidt operator.
\end{lemma}
\begin{proof}
    First, we show that $T_{I,J}$ is bounded. By definition
    \begin{align*}
     ||T_{I,J}f||^2 =&\int_{\mathbb{R}} \left|\int_{I} f(t)\overline{K_x(t)}(1+t^2)dt + \int_{J} \widehat{f}(t)\overline{\widehat{K_x}(t)}(1+t^2)dt\right|^2(1+x^2) dx \\
    + &\int_{\mathbb{R}} \left|\int_{I} f(t)\widehat{K_x}(t)(1+t^2)dt + \int_{J} \widehat{f}(t)\overline{K_x(t)}(1+t^2)dt\right|^2(1+x^2) dx.
\end{align*}
    Using $|z|^2 = z\overline{z}$, multiplying out and using the reproducing property of the kernel we obtain
\begin{align*}
    ||T_{I,J}f||^2  &=\int_I \int_I f(t)\overline{f(y)}K_y(t) (1+y^2)(1+t^2) dydt +\int_J \int_J \widehat{f}(t) \overline{\widehat{f}(y)} K_y(t) (1+y^2)(1+t^2) dydt\\
    &+ \int_I \int_J f(t)\overline{\widehat{f}(y)} \widehat{K_y}(t)(1+y^2)(1+t^2)dydt + \int_J \int_I \widehat{f}(t) \overline{{f}(y) \widehat{K_y}(t)}(1+y^2)(1+t^2)dydt.
\end{align*}
Since the kernel and its Fourier transform are bounded, the above is bounded and
$$||T_{I,J}f||\leq c\max{(|I|, |J|)}^3 ||f||.$$
    To show that $T_{I,J}$ is Hilbert--Schmidt, we compute $\sum_n ||T_{I,J}e_n||^2$, where $\{e_n\}_n$ is the orthonormal basis of $\mathcal{H}$ from Theorem \ref{ONB}. Doing the same as above with the additional facts that $\innerproduct{\widehat{K_y}}{\widehat{K_t}}=\innerproduct{K_y}{K_t}=K_y(t)$, $\widehat{e_n}(t) \overline{\widehat{e_n}(y)} = (-i)^n e_n(t) \overline{(-i)^n}e_n(y) = e_n(t)e_n(y)$, $\sum_n e_n(t)e_n(y) = K_y(t)$ and $\sum_n e_n(t)\widehat{e_n}(y)=\widehat{K_y}(t)$ gives
\begin{align*}
    \sum_n ||T_{I,J}e_n||^2&= \int_I \int_I K_y(t)^2 (1+y^2)(1+t^2) dydt + \int_J \int_J K_y(t)^2 (1+y^2)(1+t^2) dydt \\
    &+ 2\int_I \int_J |\widehat{K_y}(t)|^2(1+y^2)(1+t^2)dydt.
\end{align*}
    As $I, J$ have finite Lebesgue measures, the above is finite, and so $T_{I,J}$ is a Hilbert--Schmidt operator.
    To prove that $T_{I,J}$ is self-adjoint, one simply writes down $\innerproduct{Tf}{g}$ and $\innerproduct{f}{Tg}$, and notices that they are equal using the reproducing property of the kernel.
\end{proof}
This lemma allows us to study the eigenvalues of $T_{I,J}$.
\begin{theorem}\label{sumeig}
Let $\lambda_n$ be the $n$th eigenvalue of $T_{I,J}$. Then
\begin{align*}
    \sum_{n} \lambda_n &= \int_I K_x(x)(1+x^2)dx + \int_J K_x(x)(1+x^2)dx,\\
    \sum_{n} \lambda_n^2 &= \int_I \int_I K_y(t)^2 (1+y^2)(1+t^2) dydt\\
    &+\int_J \int_J K_y(t)^2 (1+y^2)(1+t^2) dydt \\
    &+ 2\int_I \int_J |\widehat{K_y}(t)|^2(1+y^2)(1+t^2)dydt.
\end{align*}
\end{theorem}
\begin{proof}
    By the previous lemma, $T_{I,J}$ is a compact self-adjoint operator, and thus admits an orthonormal basis $\{f_n\}_n$ of $\mathcal{H}$ such that $T_{I,J}f_n = \lambda_n f_n$. As $T_{I,J}$ is also Hilbert--Schmidt, $\sum_n ||T_{I,J}e_n||^2 = \sum_n ||T_{I,J}f_n||^2$ for any two orthonormal bases $\{e_n\}_n, \{f_n\}_n$, and so the second formula follows from the proof of the previous lemma as
\begin{align*}
    \sum_n \lambda_n^2 &= \sum_n \lambda_n^2 ||f_n||^2 = \sum_n ||T_{I,J}f_n||^2 = \sum_n ||T_{I,J}e_n||^2.
\end{align*}
For the first formula, we simply compute that 
\begin{align*}
    \lambda_n = \innerproduct{T_{I,J}e_n}{e_n} &= \int_I \innerproduct{K_x}{e_n} e_n(x)(1+x^2)dx+\int_J \innerproduct{K_x}{\widehat{e_n}}  \widehat{e_n}(x)(1+x^2)dx\\
    &= \int_I \overline{e_n}(x) e_n(x)(1+x^2)dx+\int_J \overline{\widehat{e_n}(x)} \widehat{e_n}(x)(1+x^2)dx,
\end{align*}
and thus 
$$\sum_n \lambda_n = \int_I K_x(x)(1+x^2)dx + \int_J K_x(x)(1+x^2)dx .$$
\end{proof} 
To get some information now about the size of both sums, we need to obtain an approximate formula for the reproducing kernel, as the original formula of Theorem \ref{repker} is difficult to use in practice. Since the kernel will clearly be very small if $y \notin (x-1,x+1)$, it is enough for us to check the kernel's behavior around the diagonal. 
\begin{lemma}\label{kerapprox}
For $y\in(x-1,x+1)$ we have
    $$K_x(y) \sim \frac{\pi}{x} e^{-\pi |x^2-y^2|}, \quad x\to \infty.$$
\end{lemma}
\begin{proof}
    First, we know from Theorem \ref{repker} that
    \begin{align*}
        K_x(y) = \sqrt{2} \pi e^{\pi(x-y)^2} \int_0^1 \frac{t^{2\pi-1/2}}{\sqrt{1-t^2}} e^{-2\pi xy \frac{1-t}{1+t}-2\pi\frac{(y-x)^2}{1-t^2}} dt.
    \end{align*}
    For simplicity assume that $y\in (x,x+1)$ (the other case is analogous).
    Since the maximum of $- xy \frac{1-t}{1+t}-\frac{(y-x)^2}{1-t^2}$ is attained at $t = \frac{x}{y}$, we take the Taylor expansions around $t=\frac{x}{y}$ and get 
    \begin{align*}
        K_x(y) &\sim \sqrt{2} \pi e^{\pi(x-y)^2} \int_0^1 \frac{\left(\frac{x}{y}\right)^{2\pi-1/2}}{\sqrt{1-\left(\frac{x}{y}\right)^2}} e^{2\pi y(x-y)-2\pi \frac{y^4}{y^2-x^2}(t-\frac{x}{y})^2} dt\\
        &\sim \frac{\pi}{\sqrt{1-\frac{x}{y}}} e^{\pi (x^2-y^2)} \int_0^1  e^{-2\pi \frac{y^4}{y^2-x^2}(t-\frac{x}{y})^2} dt\\
        &\sim \frac{\pi}{\sqrt{1-\frac{x}{y}}} e^{\pi (x^2-y^2)} \sqrt{\frac{y^2-x^2}{2y^4}} \sim \frac{\pi}{x} e^{\pi (x^2-y^2)}.
    \end{align*}
\end{proof}
\begin{lemma}
We have
\begin{align*}
     \widehat{K_x}(y) \sim \frac{e^{-2\pi ixy}}{x^2+y^2}, \quad x^2+y^2\to \infty.
\end{align*}
\end{lemma}
\begin{proof}
Theorem \ref{ftrepker} gives us that
\begin{align*}
    \widehat{K_x}(y) &= \sqrt{2} \pi \int_0^1 \frac{t^{2\pi-1/2}}{\sqrt{1+t^2}} e^{\pi \frac{-4xyti+(x^2+y^2)(t^2-1)}{1+t^2}} dt. 
\end{align*}
We notice that the mass of the integrand is concentrated around $t=1$. By using suitable Taylor expansions around $t=1$, we therefore obtain
\begin{align*}
     \widehat{K_x}(y) &\sim \sqrt{2} \pi \int_0^1 \frac{1}{\sqrt{2}} e^{-2\pi ixy} e^{\pi (x^2+y^2)(t-1)} dt \\
    &\sim  \frac{e^{-2\pi ixy}}{x^2+y^2}.
\end{align*}
\end{proof}
This now allows us to study the eigenvalues of $T_{I,J}$.
\begin{theorem}
    For $I=[-R_1,R_1]$, $J=[-R_2,R_2]$ we have the following for the trace of $T_{I,J}$:
    $$\sum_n \lambda_n \sim \pi R_1^2+\pi R_2^2, \quad R_1,R_2 \to \infty.$$
\end{theorem}
\begin{proof}
    By Lemma \ref{kerapprox}, $K_x(x) \sim \frac{\pi}{x}$. Thus, the conclusion follows trivially from Theorem \ref{sumeig}.
\end{proof}
\begin{theorem} \label{sumsq}
For $I=[-R_1,R_1]$, $J=[-R_2,R_2]$ we have the following for the sum of squares of eigenvalues of $T_{I,J}$ as $R_1,R_2 \to \infty$:
    \begin{align*}
        \sum_n \lambda_n^2 \sim &\frac{\pi}{2} R_1^2 + \frac{\pi}{2} R_2^2 +2((R_1^2 + R_2^2) \arctan(R_1/R_2) - R_1 R_2  \\
        &- R_1^2 \arctan(R_1) + 2 R_1^2 \arctan(R_2/R_1)).
    \end{align*}
\end{theorem}
\begin{proof}
First, 
\begin{align*}
    &\int_I \int_I K_y(t)^2 (1+y^2)(1+t^2) dydt \sim 4\int_1^{R_1} \int_{t}^{t+1} K_y(t)^2 (1+y^2)(1+t^2) dydt \\ 
    &\sim 4\int_1^{R_1} \int_{t}^{t+1} \frac{\pi^2}{yt} e^{2\pi (t^2-y^2)} (1+y^2)(1+t^2) dydt 
    \sim \pi \int_1^{R_1} t (1-e^{-2\pi(2t+1)})dt \sim \frac{\pi}{2} R_1^2. 
\end{align*}
Second, using $\int \frac{t^2}{(y^2+t^2)^2}dt = \frac{1}{2} \left(\frac{\arctan (\frac{t}{y})}{y} -\frac{t}{t^2+y^2}\right)$, $\int y \arctan(\frac{R}{y})dy = \frac{1}{2}((R^2+y^2)\arctan(\frac{R}{y})+Ry)$ and $\int\frac{Ry^2}{R^2+y^2}dy = R(y-R\arctan(\frac{y}{R})))$, we have
\begin{align*}
    &\int_I \int_J |\widehat{K_y}(t)|^2(1+y^2)(1+t^2)dydt \sim 4 \int_1^{R_1} \int_1^{R_2} \frac{1}{(y^2+t^2)^2}(1+y^2)(1+t^2)dydt\\
    &\sim (R_1^2 + R_2^2) \arctan(R_1/R_2) - R_1 R_2 + R_1 + R_2 - R_1^2 \arctan(R_1) \\
    &- R_2^2 \arctan(1/R_2) + 2 R_1^2 \arctan(R_2/R_1) - 2 R_1^2 \arctan(1/R_1).
\end{align*}
Thus, the conclusion follows from Theorem \ref{sumeig}.
\end{proof}
As we can see now, if $I=J=[-R,R]$, then the above simplifies to 
\begin{align*}
    \sum_n \lambda_n &\sim 2\pi R^2, \\
    \sum_n \lambda_n^2 &\sim  (2\pi-2)R^2.
\end{align*}
This means that the eigenvalues do not decay as quickly as we initially expected them to do. It seems like concentrating the sum of both the function and its Fourier transform is the reason for that. We do not think it has anything to do with the weight in the norm - as we can see below, in the case of a weighted Paley--Wiener space, the general behavior of eigenvalues did not change.
This leaves us with the question of how many of the eigenvalues are close to $1$, partially answered in Theorem \ref{4R^2th}.

To prove this theorem, we will consider a subspace of $\mathcal{H}$, the weighted Paley--Wiener space. This space is easier to study and will already give us the desired number of large eigenvalues. Thus, we divert our focus to this space for now, before coming back to the proof of Theorem \ref{4R^2th}. 

\section{Eigenvalue problem for the Paley--Wiener space with $(1+x^2)$ weight}\label{PW}
Consider the Paley--Wiener space $PW_{T}(1+x^2)$ of functions $f$ with Fourier transform supported on $[-T, T]$ and with the norm
$$||f||^2 = \int_{\mathbb{R}} |f(x)|^2 (1+x^2)dx <\infty.$$
Note that it is a subspace of our space $\mathcal{H}$.
It is a reproducing kernel Hilbert space, and to find the kernel $L_t(y)$ we want to solve
\begin{align*}
    f(t) = \int_{\mathbb{R}} f(y) \overline{L_t(y)} (1+y^2)dy.   
\end{align*}
Using Plancherel's theorem we obtain the differential equation
\begin{align*}
    \widehat{L_t}(y)-\frac{\widehat{L_t}^{''}(y)}{4\pi^2}=e^{-2\pi i yt}  \mathbbm{1}_{[-T,T]}(y), \quad \widehat{L_t}(-T)=\widehat{L_t}(T)=0,
\end{align*}
and solving it we get 
\begin{align*}
    \widehat{L_t}(y) = \frac{e^{2\pi T}(e^{-2\pi i tT}e^{4\pi T}-e^{2\pi itT})(e^{2\pi y}-e^{2\pi(2T-y)})}{(1 - e^{8\pi T})(1 + t^2)} + \frac{e^{-2\pi i ty}-e^{-2\pi i tT}e^{2\pi(T-y)}}{1+t^2},
\end{align*}
which in turn gives
\begin{align*}
    L_t(y) 
    = \frac{(1+ty)\sin(2\pi T(t-y))}{\pi (t-y) (1+t^2)(1+y^2)}+\frac{(1+e^{8\pi T})\cos(2\pi T(t-y))-2e^{4\pi T}\cos(2\pi T(t+y))}{\pi (1-e^{8\pi T})(1+t^2)(1+y^2)}.
\end{align*}
If we consider now a problem of maximizing the ratio
\begin{align*}
    \lambda = \frac{\int_{-R}^R |f(x)|^2 (1+x^2)dx}{\int_{\mathbb{R}} |f(x)|^2 (1+x^2)dx},
\end{align*}
we will get the corresponding compact trace class operator 
\begin{align*}
    W_Rf(t) = \int_{-R}^R f(y)\overline{L_t(y)}(1+y^2)dy.
\end{align*}
The eigenvalues $\lambda_n$ of $W_R$, sorted as usual in decreasing order, satisfy
\begin{align*}
    &\sum_n \lambda_n = \int_{-R}^R L_t(t)(1+t^2)dt,\\
    &\sum_n \lambda_n^2 = \int_{-R}^R\int_{-R}^R |L_t(y)|^2(1+t^2)(1+y^2)dtdy. 
\end{align*}
Since 
\begin{align*}
    L_t(t) &= \frac{2T}{1+t^2}+ \frac{(1+e^{8\pi T})-2e^{4\pi T}\cos(4\pi Tt)}{\pi (1-e^{8\pi T})(1+t^2)^2},
\end{align*}
we see that 
\begin{align*}
    \sum_n \lambda_n = 4RT +O\left(\frac{1}{1+R}\right).
\end{align*}
Now,
\begin{align*}
    \sum_n \lambda_n^2 &= \int_{-R}^R\int_{-R}^R |L_t(y)|^2(1+t^2)(1+y^2)dtdy\\
    &\geq \int_{-R}^R\int_{-R}^R \left(\frac{(1+ty)\sin(2\pi T(t-y))}{\pi (t-y) (1+t^2)(1+y^2)}\right)^2(1+t^2)(1+y^2)dtdy\\
    &-\int_{-R}^R\int_{-R}^R \left(\frac{c}{(1+y^2) (1+t^2)}\right)^2(1+t^2)(1+y^2)dtdy\\
    &\geq \frac{1}{\pi^2}\int_{-R}^R\int_{-R}^R \left(\frac{\sin(2\pi T(t-y))}{t-y}\right)^2\frac{(1+ty)^2}{(1+t^2)(1+y^2)}dtdy-\frac{c}{(1+R)^2}\\
    &\geq \frac{4}{\pi^2}\int_{4}^{R}\int_{y/2}^{y} \left(\frac{\sin(2\pi T(t-y))}{t-y}\right)^2\left(1-\frac{1}{2ty}\right)dtdy-\frac{c}{(1+R)^2}\\
    &\geq \frac{4}{\pi^2}\int_{4}^{R}\left(\pi^2T-\frac{2}{y}-\frac{\pi^2T-\frac{2}{y}}{y^2} \right)dy-\frac{c}{(1+R)^2}\\
    &\geq 4RT-O(\log R)-O(T).
\end{align*}

Using standard arguments that take into account the fact that $0< \lambda_n <1$ and considering $\sum_n \lambda_n(1-\lambda_n)$, the above facts prove Theorem \ref{PWcon}. We refer to the proof of Theorem \ref{onesideth}, which will carry out an analogous computation.

\section{Eigenvalue problem for one-sided concentration}
Intrigued by the behavior of the eigenvalues of $T_{I,J}$, we study what would happen in the case of one-sided concentration, which leads to Theorem \ref{onesideth}.
Here we consider the problem of maximizing the ratio
$$\lambda = \frac{\int_{I} |f(x)|^2(1+x^2) dx}{\int_{\mathbb{R}} |f(x)|^2(1+x^2) dx+\int_{\mathbb{R}} |\hat{f}(\xi)|^2(1+\xi^2) d\xi},$$
where $I\subset \mathbb{R}$ is measurable.
This leads to the eigenvalue problem
$$\lambda_n f_n(t) = \int_{I} f_n(x)\overline{K_t(x)}(1+x^2)dx ,$$
where $K_t$ is the reproducing kernel of $\mathcal{H}$.
Define an operator $T_I:\mathcal{H} \to \mathcal{H}$ as
$$T_If(t) = \int_{I} f(x)\overline{K_t(x)}(1+x^2)dx.$$
By computations analogous to the two-sided case, we obtain
\begin{align*}
    &\sum_n ||T_Ie_n||^2 = \int_I \int_I K_y(t)^2 (1+y^2)(1+t^2) dydt
\end{align*}
and
\begin{align*}
    \sum_n \innerproduct{T_Ie_n}{e_n} =\sum_n \int_I \overline{e_n(x)} e_n(x)(1+x^2)dx = \int_I K_x(x)(1+x^2)dx.
\end{align*}
Thus if $I=[-R,R]$, then as $R\to \infty$,
\begin{align*}
    &\sum_n \lambda_n \sim \pi R^2, \\
    &\sum_n \lambda_n^2 \sim \frac{\pi}{2} R^2
\end{align*}
by the computation done in the proof of Theorem \ref{sumsq}.
Now, to obtain the third moment needed for the proof of Theorem \ref{onesideth}, we analyze
\begin{align*}
    T_I^3f(t) = \int_{I} \int_{I} \int_{I} f(x_1)\overline{K_{x_2}(x_1)}(1+x_1^2)dx_1\overline{K_{x_3}(x_2)}(1+x_2^2)dx_2\overline{K_t(x_3)}(1+x_3^2)dx_3,
\end{align*}
which gives
\begin{align*}
     &\innerproduct{T_I^3 e_n}{e_n} = \int_{I} \int_{I} \int_{I} e_n(x_1)K_{x_2}(x_1)K_{x_3}(x_2) \overline{e_n(x_3)}(1+x_1^2)(1+x_2^2)(1+x_3^2) dx_1dx_2dx_3 ,
\end{align*}
and thus
\begin{align*}
    \sum_n \innerproduct{T_I^3 e_n}{e_n} = \int_{I} \int_{I} \int_{I} K_{x_1}(x_3)K_{x_2}(x_1)K_{x_3}(x_2)(1+x_1^2)(1+x_2^2)(1+x_3^2) dx_1dx_2dx_3.
\end{align*}
Using that mass will be concentrated around the diagonal and using the approximation formula for the kernel (Lemma \ref{kerapprox}), we find that as $R\to \infty$,
\begin{align*}
    &\int_{I} \int_{I} \int_{I} K_{x_1}(x_3)K_{x_2}(x_1)K_{x_3}(x_2)(1+x_1^2)(1+x_2^2)(1+x_3^2) dx_1dx_2dx_3  \\
    &\sim 2^3 \int_1^R \int_{x_1}^{x_1+1}\int_{x_2}^{x_2+1} K_{x_1}(x_3)K_{x_2}(x_1)K_{x_3}(x_2)x_1^2 x_2^2 x_3^2 dx_3dx_2dx_1 \\
    &\sim (2\pi)^3 \int_1^R \int_{x_1}^{x_1+1}\int_{x_2}^{x_2+1} e^{-2\pi (x_3^2-x_1^2)} x_1x_2x_3 dx_3dx_2dx_1 \\
    &\sim \frac{(2\pi)^3}{4\pi} \int_1^R \int_{x_1}^{x_1+1} e^{-2\pi (x_2^2-x_1^2)} x_1x_2dx_2dx_1 \\
    &\sim \frac{(2\pi)^3}{(4\pi)^2} \int_1^R   x_{1} dx_1 \sim \frac{(2\pi)^3}{(4\pi)^{2}} \frac{R^2}{2} = \frac{\pi}{4} R^2.
\end{align*}
We are now prepared to prove Theorem \ref{onesideth}.
\begin{proof}[Proof of Theorem \ref{onesideth}]
    Consider the sum $\sum_n \lambda_n(\lambda_n-\frac{1}{2})^2.$ By the previous computations of the moments it is
    \begin{align*}
        \sum_n \lambda_n\left(\lambda_n-\frac{1}{2}\right)^2 &= \sum_n \left(\lambda_n^3-\lambda_n^2+\frac{1}{4}\lambda_n\right) \\
        &= \frac{\pi}{4}R^2-\frac{\pi}{2}R^2 +\frac{\pi}{4}R^2+o(R^2) = o(R^2).
    \end{align*}
    Since all eigenvalues are necessarily between $0$ and $1$, this already shows that there are at most $o(R^2)$ eigenvalues larger than $\frac{1}{2}+\epsilon$ and between $\epsilon$ and $\frac{1}{2}-\epsilon$.
    Now, it is clear that there are at most $2\pi R^2+o(R^2)$ eigenvalues between $\frac{1}{2}-\epsilon$ and $\frac{1}{2}+\epsilon$ (as $\sum_n \lambda_n = \pi R^2+o(R^2)$). If there would be only $(2\pi-\delta)R^2$ of them for some $\delta>0$, then to still have $\sum_n \lambda_n \sim \pi R^2$ we would need to have $\sum_{\lambda_n <\epsilon} \lambda_n \sim \delta R^2$ (as we already know that there are $o(R^2)$ eigenvalues greater than $\frac{1}{2}+\epsilon$ and between $\epsilon$ and $\frac{1}{2}-\epsilon$). But then we would have $\sum_{\lambda_n < \epsilon} \lambda_n^2 <\epsilon \delta R^2$, $\sum_{\frac{1}{2}-\epsilon<\lambda_n <\frac{1}{2}+\epsilon} \lambda_n^2 < (\frac{1}{2}+\epsilon)^2 (2\pi-\delta) R^2 < \frac{1}{4}(2\pi-\delta)R^2,$ and so in total $\sum_n \lambda_n^2$ would be smaller than $(\frac{\pi}{2}-\frac{\delta}{4}+\epsilon \delta)R^2 + o(R^2)$. This is a contradiction with $\sum_n \lambda_n^2 = \frac{\pi}{2}R^2+o(R^2)$ for any $\epsilon<\frac{1}{4}$, which concludes the proof.
\end{proof}

\section{Proof of Theorem \ref{4R^2th}}
\begin{proof}[Proof of Theorem \ref{4R^2th}]
    Let $\epsilon>0$, $I=[-R,R], J=[-T,T]$. We will first prove that there are at least $4RT - O(\log R)-O(T)$ eigenvalues of $T_{I,J}$ larger than $1-\epsilon$.
    Define $\mathcal{H}_N$ as $\spn\{f_n\}_{n=1}^{N(R,T)}$, where $f_n$ are eigenfunctions of $W_R$ coming from the weighted Paley--Wiener space $PW_{T}(1+x^2)$, and where $N(R,T)$ denotes the number of eigenvalues of $W_R$ larger than $1-\epsilon$. Clearly, for any $f\in \mathcal{H}_N$ we have
    \begin{align*}
        \frac{\int_{-R}^R |f(x)|^2 (1+x^2)dx}{\int_{\mathbb{R}} |f(x)|^2 (1+x^2)dx}>1-\epsilon.
    \end{align*}
    All functions from $\mathcal{H}_N$ are of course also in $\mathcal{H}$ and
    \begin{equation}\label{epscon}
    \begin{split}
       &\frac{\int_{-R}^{R}|f(x)|^2(1+x^2)dx+\int_{-T}^{T}|\hat{f}(x)|^2(1+x^2)dx}{\int_{\mathbb{R}}|f(x)|^2(1+x^2)dx+\int_{\mathbb{R}}|\hat{f}(x)|^2(1+x^2)dx} \\
        &\geq \frac{(1-\epsilon)\int_{\mathbb{R}}|f(x)|^2(1+x^2)dx+\int_{\mathbb{R}}|\hat{f}(x)|^2(1+x^2)dx}{\int_{\mathbb{R}}|f(x)|^2(1+x^2)dx+\int_{\mathbb{R}}|\hat{f}(x)|^2(1+x^2)dx}\\
        &\geq 1-\epsilon.
    \end{split}
    \end{equation}
    We therefore have an $N(R,T)$ dimensional subspace of $\mathcal{H}$ of functions satisfying (\ref{epscon}), and so the conclusion follows by Theorem \ref{PWcon} and the min-max theorem. 

    To prove that there are at most $2\pi(R^2+T^2) + o(R^2)+o(T^2)$ eigenvalues of $T_{I,J}$ larger than $\epsilon<\frac{1}{4}$, we notice first that a theorem analogous to Theorem \ref{onesideth} holds for the operator $\widetilde{T_J}:\mathcal{H} \to \mathcal{H}$ defined as
    $$\widetilde{T_J}f(t) = \int_{J} \hat{f}(x)\overline{\hat{K_t}(x)}(1+x^2)dx,$$
    corresponding to the problem of maximizing the ratio 
    \begin{align*}
    \frac{\int_{J} |\hat{f}(x)|^2(1+x^2) dx}{\int_{\mathbb{R}} |f(x)|^2(1+x^2) dx+\int_{\mathbb{R}} |\hat{f}(\xi)|^2(1+\xi^2) d\xi}.
    \end{align*}
    With this notation, we have $T_{I,J}f(t) = T_If(t)+\widetilde{T_J}f(t)$. 
    Let $\{g_n(x)\}_{n\geq0}$ and $\{\widetilde{g_n}(x)\}_{n\geq0}$ be two orthonormal bases of $\mathcal{H}$, where $T_{I}g_n(x) = \lambda_ng_n(x)$ and $\widetilde{T_{J}}\widetilde{g_n}(x) = \widetilde{\lambda_n}\widetilde{g_n}(x)$, and where $\lambda_0\geq\lambda_1\geq\lambda_2\geq\dots$, $\widetilde{\lambda_0}\geq\widetilde{\lambda_1}\geq\widetilde{\lambda_2}\geq\dots$.
    Note that we have a double orthogonality of such bases with respect to the studied concentration, i.e. for $n\neq m$ 
    \begin{align*}
        0 &= \lambda_n\innerproduct{g_n}{g_m} = \innerproduct{T_Ig_n}{g_m} = \int_{\mathbb{R}}T_Ig_n(x)\overline{g_m(x)}(1+x^2)dx+\int_{\mathbb{R}}\widehat{T_Ig_n}(x)\overline{\widehat{g_m}(x)}(1+x^2)dx \\
        &=\int_I g_n(x)\overline{g_m(x)}(1+x^2)dx.
    \end{align*}
    This implies that if $f(x)=\sum_{n=0}^{\infty}c_ng_n(x)$ is such that $\frac{\int_I |f(x)|^2(1+x^2)dx}{||f||^2}>\epsilon$, then $c_N\neq 0$ for some $N$ with $\lambda_N>\epsilon$. Let us thus construct 
    \begin{align*}
        S := \spn \{g_n, \widetilde{g_m}|\lambda_n,\widetilde{\lambda_m}>\frac{\epsilon}{2}\}, 
    \end{align*}
    where $\lambda_n,\widetilde{\lambda_m}$ are the corresponding eigenvalues. Using the min-max theorem for this subspace will therefore give
    \begin{align*}
        \mu_n < \max_{f\in S^{\perp}} \frac{\innerproduct{T_{I,J}f}{f}}{||f||^2}<\epsilon,
    \end{align*}
    where $\mu_n$ denotes the $n$th eigenvalue of $T_{I,J}$.
    Theorem \ref{onesideth} implies that the dimension of this subspace is at most $2\pi R^2 + 2\pi T^2 + o(R^2)+o(T^2)$, which concludes the proof. 
\end{proof}

\thanks{\textbf{Acknowledgments} I would like to extend my gratitude to my Ph.D. supervisor Kristian Seip for his guidance and for introducing me to the topics presented in this paper.
Special thanks to Joaquim Bruna, Hans G. Feichtinger, Aleksei Kulikov, and Joaquim Ortega Cerdà for fruitful discussions and for their comments.}
\bibliographystyle{abbrv}
\bibliography{papers}

@article{lyubarskii_convergence_1999,
  author = "{L}yubarskii, {Y}urii and {S}eip, {K}ristian ",
  year = "1999",
  title = "{C}onvergence and summability of {G}abor expansions at the {N}yquist density",
  journal = "{J}. {F}ourier {A}nal. {A}ppl.",
  volume = "5",
  number = "2-3",
  pages = "127-157",
  DOI = "10.1007/{B}{F}01261606",
}

@article {kulikov2023fourier,
    AUTHOR = {Kulikov, Aleksei and Nazarov, Fedor and Sodin, Mikhail},
     TITLE = {{F}ourier uniqueness and non-uniqueness pairs},
   JOURNAL = {J. Math. Phys. Anal. Geom.},
  FJOURNAL = {Journal of Mathematical Physics, Analysis, Geometry},
    VOLUME = {21},
      YEAR = {2025},
    NUMBER = {1},
     PAGES = {84--130},
      ISSN = {1812-9471,1817-5805},
   MRCLASS = {42A38 (42B10)},
  MRNUMBER = {4866154},
}

@article{Watson,
author = {Watson, G. N.},
title = {Notes on Generating Functions of Polynomials: (2) {H}ermite Polynomials},
journal = {J. Lond. Math. Soc.},
volume = {s1-8},
number = {3},
pages = {194-199},
doi = {https://doi.org/10.1112/jlms/s1-8.3.194},
year = {1933}
}

@article{Landau,
  author = "{L}andau, {H}. ",
  year = "1967",
  title = "{S}ampling, data transmission, and the {N}yquist rate",
  journal = "{P}roc. {I}{E}{E}{E}",
  volume = "55",
  number = "10",
  pages = "1701-1706",
  organization = "others",
}

@unpublished{master,
author={Zelent, Denis},
title={The reproducing kernel of the {F}ourier symmetric {S}obolev space},
school={NTNU},
year={2024},
note={Master thesis},
url={https://hdl.handle.net/11250/3138281}
}

@article{Slepian,
  author = "{S}lepian, {D}. ",
  year = "1983",
  title = "{S}ome comments on {F}ourier {A}nalysis and {U}ncertainty and {M}odelling",
  journal = "{S}{I}{A}{M} {R}ev.",
  volume = "25/{N}r.3",
  pages = "379--393",
  organization = "{C}lassical",
}

@article{LANDAU1980469,
  author = "{L}andau, {H}enry {J}. and {W}idom, {H}arold ",
  year = "1980",
  title = "{E}igenvalue distribution of time and frequency limiting.",
  journal = "{J}. {M}ath. {A}nal. {A}ppl.",
  volume = "77",
  pages = "469-481",
  organization = "others",
}

@article{KULIKOV2024101639,
  author = "{K}ulikov, {A}leksei ",
  year = "2024",
  title = "{E}xponential lower bound for the eigenvalues of the time-frequency localization operator before the plunge region",
  journal = "{A}ppl. {C}omput. {H}armon. {A}nal.",
  volume = "71",
  pages = "101639",
  publisher = "{E}lsevier",
}

@book{gr01,
  author = "{G}r{\"o}chenig, {K}arlheinz ",
  year = "2001",
  title = "{F}oundations of {T}ime-{F}requency {A}nalysis",
  series = "{A}ppl. {N}umer. {H}armon. {A}nal.",
  pages = "xvi+359",
  publisher = "{B}irkh{\"a}user",
  organization = "{N}u{H}{A}{G}",
}

@book{shubin,
  author = "{S}hubin, {M}ikhail {A}. ",
  year = "2001",
  title = "{P}seudodifferential {O}perators and {S}pectral {T}heory. {T}ransl. from the {R}ussian by {S}tig {I}. {A}ndersson. 2nd ed.",
  publisher = "{S}pringer",
  organization = "others",
}

@article{boggiatto2004,
  author = "{B}oggiatto, {P}aolo and {C}ordero, {E}lena and {G}r{\"o}chenig, {K}arlheinz ",
  year = "2004",
  title = "{G}eneralized anti-{W}ick operators with symbols in distributional {S}obolev spaces",
  journal = "{I}ntegr. {E}qu. {O}per. {T}heory",
  volume = "48",
  number = "4",
  pages = "427--442",
  organization = "{N}u{H}{A}{G}",
}

@article{szehr25,
  author = "{S}zehr, {O}leg ",
  year = "2025",
  title = "{S}pectral Criteria for Unique Signal Recovery from Two-Sided Sampling",
  journal = "ar{X}iv preprint ar{X}iv:2509.14953",
}

@article{Luef11,
  author = "{L}uef, {F}ranz and {R}ahbani, {Z}ohreh ",
  year = "2011",
  title = "{O}n pseudodifferential operators with symbols in generalized {S}hubin classes and an application to {L}andau-{W}eyl operators",
  journal = "{B}anach {J}. {M}ath. {A}nal.",
  volume = "5",
  number = "2",
  pages = "59-72",
  organization = "{N}u{H}{A}{G}",
}

@article{Fei02,
  author = "{D}e {M}ari, {F}ilippo and {F}eichtinger, {H}ans {G}. and {N}owak, {K}rzysztof ",
  year = "2002",
  title = "{U}niform eigenvalue estimates for time-frequency localization operators",
  journal = "{J}. {L}ondon {M}ath. {S}oc.",
  volume = "65",
  number = "3",
  pages = "720-732",
  organization = "{N}u{H}{A}{G}",
}

@article{Fei01,
  author = "{F}eichtinger, {H}ans {G}. and {N}owak, {K}rzysztof ",
  year = "2001",
  title = "{A} {S}zeg{\"o}-type theorem for {G}abor-{T}oeplitz localization operators",
  journal = "{M}ichigan {M}ath. {J}.",
  volume = "49",
  number = "1",
  pages = "13-21",
  organization = "{N}u{H}{A}{G};{C}lassical",
}

@incollection{Fei14,
  author = "{F}eichtinger, {H}ans {G}. and {N}owak, {K}rzysztof and {P}ap, {M}argit ",
  year = "2014",
  title = "{S}pectral properties of {T}oeplitz operators acting on {G}abor type reproducing kernel {H}ilbert spaces",
  booktitle = "{M}athematics without {B}oundaries: {S}urveys in {P}ure {M}athematics",
  publisher = "{S}pringer",
  DOI = "https://doi.org/10.1007/978-1-4939-1106-6_6",
}

@article{Gro13,
  author = "{G}r{\"o}chenig, {K}arlheinz and {T}oft, {J}oachim ",
  year = "2013",
  title = "{T}he range of localization operators and lifting theorems for modulation and {B}argmann-{F}ock spaces",
  journal = "{T}rans. {A}mer. {M}ath. {S}oc.",
  volume = "365",
  pages = "4475-4496",
  organization = "{N}u{H}{A}{G}",
}

@article{boto05,
  author = "{B}oggiatto, {P}aolo and {T}oft, {J}oachim ",
  year = "2005",
  title = "{E}mbeddings and compactness for generalized {S}obolev-{S}hubin spaces and modulation spaces.",
  journal = "{A}ppl. {A}nal.",
  volume = "84",
  number = "3",
  pages = "269-282",
  organization = "{G}abor",
}

\end{document}